\documentclass{article}
\usepackage{amsfonts,amsmath,amssymb,mathtools,amsthm}
\usepackage{authblk}
\usepackage{listings}
\usepackage{graphicx}
\usepackage{booktabs}
\usepackage{tcolorbox}
\usepackage{url}

\usepackage[text={6in,9in},centering]{geometry}

\def\bbbone{{\mathchoice {\rm 1\mskip-4mu l} {\rm 1\mskip-4mu l}
{\rm 1\mskip-4.5mu l} {\rm 1\mskip-5mu l}}}
\def\bbbe{\mathbb{E}}
\def\bbbp{\mathbb{P}}

\def\ignore#1{}

\def\tC{\tilde{C}}
\def\tm{\tilde{m}}
\def\ts{\tilde{s}}
\def\tK{\tilde{K}}
\def\tZ{\tilde{Z}}
\def\tN{\tilde{N}}
\def\tlam{\tilde{\lambda}}

\def\BP{\mathbb{P}}
\def\BE{\mathbb{E}}
\def\BN{\mathbb{N}}
\def\no{\noindent}
\def\eq{\begin{equation}}
\def\en{\end{equation}}

\def\bZ{\mbox{\bf{Z}}}
\def\bC{\mbox{\bf{C}}}

\newtheorem{lemma}{Lemma}
\newtheorem{remark}{Remark}

\title{A note on the Screaming Toes game} % insert title - use \\ if it requires more than one line.

\author[1]{Simon Tavar\'e}

\affil[1]{Department of Statistics, Columbia University, 1255 Amsterdam Avenue, New York, NY 10027, USA} % Your postal address goes here.

\begin{document}
\maketitle

\begin{abstract}
We investigate properties of random mappings whose core is composed of derangements as opposed to permutations. Such mappings arise as the natural framework to study the Screaming Toes game described, for example, by Peter Cameron. This mapping differs from the classical case primarily in the behaviour of the small components, and a number of explicit results are provided to illustrate these differences. 
\end{abstract}



\section{Introduction}  
The following problem comes from Peter Cameron's book, \cite[p. 154]{cameron_notes_2017}.
\begin{quotation}
\noindent \emph{$n$ people stand in a circle. Each player looks down at someone else's feet (i.e., not at their own feet). At a given signal, everyone looks up from the feet to the eyes of the person they were looking at. If two people make eye contact, they scream. What is the probability $q_n$, say, of at least one pair screaming?}
\end{quotation}
The purpose of this note is to put this problem in its natural probabilistic setting, namely that of a random mapping whose core is a derangement, for which many properties can be calculated simply. We focus primarily on the small components, but comment on other limiting regimes in the discussion.

We begin by describing the usual model for a random mapping. Let $B_1,B_2,\ldots,B_n$ be independent and identically distributed random variables satisfying
\begin{equation}\label{Bdef}
\mathbb{P}(B_i = j) = 1/n, \quad j \in [n],
\end{equation}
where $[n] = \{1,2,\ldots,n\}$. The mapping $f: [n] \to [n]$ is given by $f(i) = B_i$.
Components of the mapping are formed by iteration: $i$ and $j$ are in the same component if some iterate of $i$ equals some iterate of $j$; each component is a directed cycle of rooted, labeled trees. An example with $n = 20$, displayed in  Fig.~\ref{fig1}, is given by \\

\hskip -0.25in
\begin{tabular}{c|cccccccccccccccccccc}
$i$ & 1 & 2 & 3 & 4 & 5 & 6 & 7 & 8 & 9 & 10 & 11 & 12 & 13 & 14 & 15 & 16 & 17 & 18 & 19 & 20 \\
$B_i$ & 2 & 14 & 7 & 1 & 7 & 19 & 17 & 11 & 10 & 13 & 2 & 14 & 9 & 8 & 19 & 10 & 6 & 16 & 6 & 19\\
\end{tabular}

\begin{figure}[htbp]
\centering
\includegraphics[width=0.6\textwidth]{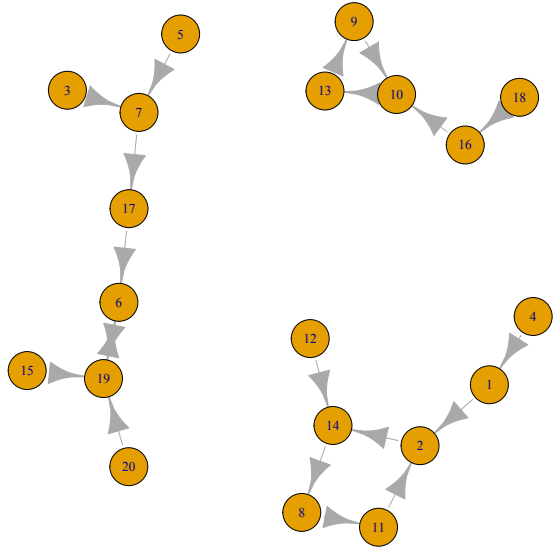}
\vskip -0.9in
\caption{A mapping graph on $n = 20$ vertices with no singleton cycles. This one has 3 components, of sizes 5, 7 and 8. There are 9 elements in cycles, which have length 3, 4 and 2 respectively. Figure produced by the {\sf R} igraph package~\cite{csardi2006}. }\label{fig1}
\end{figure}

\subsection{The components of a random mapping}
Denoting the number of components of size $j$ by $C_j(n), j=1,2, \ldots, n$, Harris~\cite{bh60} showed that the  probability that a random mapping  has $a_j$
components of size $j$ is 
\begin{equation}\label{hsf}
\BP(C_j(n) = a_j, j=1,\ldots,n) = \bbbone\left\{\sum_{j=1}^n j a_j = n\right\}\, \frac{n!e^n}{n^n} \prod_{j=1}^n 
\frac{\lambda_j^{a_j}}{a_j!},
\end{equation}
where  
\begin{equation}\label{lambdaj}
\lambda_j = \frac{e^{-j}}{j} \sum_{i=0}^{j-1}\frac{j^i}{i!} = \frac{1}{j}\,\bbbp(\textrm{Po}(j) < j),
\end{equation}
 $\textrm{Po}(\mu)$ denoting a Poisson random variable with mean $\mu$. In particular, the probability that a mapping of size $n$ has a single component is
 $$%\eq\label{singledef}
 s_n := \bbbp(C_n(n) = 1) = \frac{n! e^n}{n^n} \, \lambda_n.
$$% \en
 It follows readily from (\ref{hsf}) that
\begin{eqnarray}\label{hsfmeans}
\BE C_j(n) & = & \frac{n!}{n^n} \frac{(n-j)^{n-j}}{(n-j)!}
e^j \lambda_j, \nonumber \\
& = & s_j \,\binom{n}{j} \left(\frac{j}{n}\right)^j \left( 1 - \frac{j}{n}\right)^{n-j}\quad j=1,2,\ldots,n.
\end{eqnarray}
A probabilistic interpretation of (\ref{hsfmeans}) is given in \cite{dep91}. Kolchin~\cite{kolchin1976} established that 
$(C_1(n),C_2(n),\ldots) \Rightarrow (Z_1,Z_2,\ldots),$ where the $Z_i$ are independent Poisson random variables with means $\BE Z_i = \lambda_i.$
Many other properties of random mappings may be found, for example, in ~\cite{dep91}.

\subsection{The core of a random mapping}
%This comes directly from~\cite{at92b}
Here we record some properties of the core of a random mapping, the set of elements that are in cycles.  Results (\ref{corenumber}) through (\ref{core1}) are classical; see, for
example,~\cite[p. 366]{bb85}. The number $N_n $ of elements in the core has  distribution 
given by  
\eq\label{corenumber} 
\BP(N_n = r) = \frac{r}{n}\prod_{l=0}^{r-1} 
\left(1 - \frac{l}{n}\right) = \frac{r}{n} \frac{n_{[r]}}{n^r}, \quad r=1,\ldots,n, 
\en 
where $n_{[j]} = n(n-1)\cdots (n-j+1).$ The mean of $N_n$ is
$$
\bbbe N_n   = \sum_{l=0}^{n-1} \frac{(n-1)_{[l]}}{n^l},
%\sum_{l=1}^{n} \BP(N_n \geq l) =  \sum_{l=1}^{n} \pi_{l-1} =
%\sum_{l=0}^{n-1} \prod_{j=0}^l \left(1 -\frac{j}{n}\right) 
$$
and it follows directly from (\ref{corenumber}) that $N_n/\sqrt{n}$ converges 
in distribution to a random variable with density function 
$x e^{-x^2/2}, x >0$. We write $C_j^*(n)$ for the number of 
cycles of size $j$ in the core of a random mapping, and let $C'_j(r)$ 
be the number of cycles of size $j$ in a  {\em uniform} random permutation 
of $r$ objects. The joint law of the ${\cal L}(C_j^*(n))$ is given by 
$$%\eq\label{lawcj*n}
{\cal L}(C_j^*(n)) = \sum_{r=1}^n \BP(N_n = r) {\cal L} (C'_j(r)), 
$$%\en
since, conditional on $N_n = r$ the random mapping restricted to its
core is a uniformly distributed permutation on those $r$ elements.
It follows that 
\eq\label{blurb}
\BE C_j^*(n) = \frac{1}{j} \frac{n_{[j]}}{n^j}, \quad j=1,\ldots,n. 
%\prod_{l=1}^{j-1} \left( 1-\frac{l}{n}\right).
\en
For fixed $j$, $\BE C_j^*(n) \to 1/j$, and 
\eq\label{core1}
(C_1^*(n),C_2^*(n),\ldots) \Rightarrow (Z_1^*,Z_2^*,\ldots)
\en
where $Z_j^*$ are independent Poisson random variables with mean $\BE Z^*_j =
1/j$. 

\section{The Screaming Toes random mapping}

We return now to Cameron's setting. Label the players $1, 2, \ldots, n$, and form components by iteration: $i$ and $j$ are in the same component if some iterate of $i$ equals some iterate of $j$; components now describe how the players are looking at each other. The cycles in the core of the mapping indicate sets of players, say $i_1, \ldots,i_r$ for which $i_1 \to i_2 \to \cdots \to i_r \to i_1$ (where $\to$ denotes `looks at the feet of'), and the trees attached to any of the cyclic elements describe the sets of players who look from one to another, and finally to someone in a cycle. Each component has a single cycle, so that the number of components is the number of cycles in the core, and each cycle must have length at least two.
Fig.~\ref{fig1} provides an illustration.

We study the cycles in the core and the structure of the components of the mapping for the screaming toes game. We find the probability that there are $k$ screaming pairs (for $k = 1, 2, \ldots, \lfloor n/2 \rfloor)$ (see (\ref{c2n*law})), identify the structure of the random mapping itself, and derive some basic properties of the component sizes.  In the setting of (\ref{Bdef}), the $B_i$ are independent, but no longer identically distributed; we let
\begin{equation}\label{Bdef1}
\mathbb{P}(B_i = j) = 1/(n-1), \quad j \in [n]\setminus \{i\}.
\end{equation}

The $B_i$ may be used simulate the modified mapping, which can be decomposed into components by iteration (as before), each of which is a directed cycle of rooted, labeled trees. In this case, though, the core of the mapping can have no cycles of length 1 and there are $p(n) = (n-1)^n$ such mappings. Note that the resulting combinatorial structure is not the same as a random mapping conditioned on having no singleton components, because such a conditioned structure may still have  singleton cycles in its core. We define $\tC_j(n)$ as the number of components of size $j$, and $\tC_j^*(n)$  the number of cycles of length $j$ in the core. 

\subsection{The distribution of the component sizes}

We adopt the general approach from~\cite[Chapter 2]{arratia_logarithmic_2003}. A component of size $i \geq 2$ has $j = 2, 3, \ldots, i$ elements in its core. A modification of the counting argument that leads to (\ref{lambdaj}) then shows that the number of components of size $i$ is given by
$$
\tm_i := \sum_{j=2}^i \genfrac(){0pt}{}{i}{j}\, (j-1)! \, j i^{i-j-1} = (i-1)! \sum_{j = 2}^i \frac{i^{i-j}}{(i-j)!}
$$
It follows that for $i \geq 2$,
$$%\eq\label{mdef}
\tm_i = (i-1)! \, e^i \, \sum_{l=0}^{i-2} \frac{e^{-i} i^l}{l!} = (i-1)! \, e^i \, \mathbb{P}({\rm Po}(i) < i-1).
$$%\en
The joint law of $(\tC_2(n),\ldots,\tC_n(n))$ is therefore given by
\eq\label{jointlaw}
\BP(\tC_j(n) = a_j, j=2,\ldots,n) = \bbbone\left\{\sum_{j=2}^n j a_j = n\right\}\, 
\frac{x^{-n} n!}{(n-1)^n}\,\prod_{j=2}^n \left(\frac{\tm_j x^j}{j!}\right)^{a_j} \frac{1}{a_j!},
\en
for any $x > 0$. Since the distribution in (\ref{jointlaw}) is independent of $x$, we are free to choose it and we make the choice $x = e^{-1}$, which results in the structure being \emph{logarithmic} in the terminology of \cite[p. 51]{arratia_logarithmic_2003}. We therefore define
\eq\label{lambdajhat}
\tlam_j = \frac{\tm_j e^{-j}}{j!} = \frac{1}{j}\,\bbbp(\textrm{Po}(j) < j-1), \quad j=2,3,\ldots
\en
which should be compared to that in (\ref{lambdaj}). The probability that a mapping of size $n$ has a single component is
\eq\label{singledeftoes}
\ts_n = \bbbp(\tC_n(n) = 1) = \frac{e^n n!}{(n-1)^n}\,\tlam_n,
\en
and
$$
\ts_n \sim e \sqrt\frac{\pi}{2}\,n^{-1/2} \approx 3.4069\, n^{-1/2}, \quad n \to \infty.
$$
\subsubsection{Moments}

The falling factorial moments of the component counts are readily calculated from (\ref{jointlaw}), to obtain
\eq\label{jointmeans}
\mathbb{E}(\tC_2^{[r_2]} \cdots \tC_b^{[r_b]}) = \tlam_2^{r_2} \cdots \tlam_b^{r_b} \, 
e^m n_{[m]} \frac{(n-m-1)^{n-m}}{(n-1)^n},
\en
%\frac{n!}{p(n)} \,\frac{p(n-m)}{(n-m)!},
%where $p(n) = e^{-n} (n-1)^n,$ and 
for $r_2, \ldots,r_b \geq 0$ satisfying $m = 2r_2 + \cdots + $ $b r_b \leq n$. It follows that, for $ j = 2,3,\ldots,n$,
\begin{eqnarray}
\mathbb{E} \tC_j(n) &  = & \tlam_j\, e^j\, n_{[j]} \frac{(n-j-1)^{n-j}}{(n-1)^n} \nonumber \\
& = & \ts_j \binom{n}{j} \left(\frac{j-1}{n-1}\right)^j\, \left( 1 - \frac{j}{n-1}\right)^{n-j}, \label{cjmeans}
\end{eqnarray}
which also admits a simple probabilistic justification. A numerical example is given in Table~\ref{tab2}. The covariances may be found from (\ref{jointmeans}): for $i+j \leq n$,
\eq\label{ecicj}
\bbbe \tC_i(n) \tC_j(n) = \ts_i \ts_j \binom{n}{i,j} \left(\frac{i-1}{n-1}\right)^i \left(\frac{j-1}{n-1}\right)^j \left(1 - \frac{i+j}{n-1}\right)^{n-i-j},
\en
the value being 0 when $i+j > n$. The expected value of the number of components $\tK_n = \tC_2(n) + \cdots + \tC_n(n)$ is
\eq\label{ekn}
 \bbbe \tK_n = \sum_{j=2}^n \tlam_j\, e^j \, n_{[j]} \frac{(n-j-1)^{n-j}}{(n-1)^n}.
\en

\subsubsection{Limit distributions}

Following \cite[p.48]{arratia_logarithmic_2003}, the joint law of an assembly such as the screaming toes mapping $(\tC_2(n),\ldots,\tC_n(n))$ may also be represented as that of $(\tZ_2,\ldots,\tZ_n)$ conditional on
\eq\label{t1ndef}%$$
T_{1n} := 2\tZ_2 + 3 \tZ_3 + \cdots + n\tZ_n = n,
\en%$$
where the $\tZ_i$ are independent Poisson random variables with $\bbbe \tZ_j = \tlam_j, j \geq 2,$ given in (\ref{lambdajhat}) and it follows from~\cite[Chapter 3]{arratia_logarithmic_2003}, or directly from (\ref{jointmeans}), that the counts of small components have, asymptotically,  independent Poisson distributions with means $\bbbe \tZ_j$ given above.

\section{The core of the Screaming Toes mapping}
The core of our mapping is composed of derangements, permutations with no fixed points. We use $^*$ to denote derangements, so that $C_j^*(n)$ is the number of cycles of length $j$ in a random uniform derangement of size $n$. We write
$$%\eq\label{derange}
D_n := n! \sum_{j=0}^n \frac{(-1)^j}{j!}
$$%\en
to denote the number of derangements of $n$ objects; the probability that a random permutation is a derangement is $D_n / n!$.

\ignore{
\begin{remark}
By counting permutations by their number of fixed points, we recover the fact that
$$
n! = \sum_{i=0}^n {n \choose i} D_{n-i}, \quad n=1,2,3,\ldots
$$
\end{remark}
\note{Reminder: it is clear that 
$$
\sum_{n \geq 0} \frac{D_n}{n!} x^n = \frac{e^{-x}}{1 - x}.
$$}
}

We record two results for future use:
\eq\label{ecjn*}
\BE C_j^*(n) = \frac{1}{j} \frac{n!}{D_n} \frac{D_{n-j}}{(n-j)!}, j = 2,3,\ldots,n.
\en
and, for a random permutation,
\eq\label{prob0k}
\bbbp(C_1(n) = 0, C_2(n) = k) 
%& = & \left(\frac{1}{2}\right)^k \frac{1}{k!}\, \sum_{l_1 + 2 l_2 \leq n - 2k}
%\, (-1)^{l_1 + l_2} \prod_{i=1}^2 \left(\frac{1}{i}\right)^{l_i}\frac{1}{l_i!} \\ 
% =  \left(\frac{1}{2}\right)^k \frac{1}{k!}\, \sum_{l = 0}^{\lfloor n/2 \rfloor - k} (-1)^{l} \left(\frac{1}{2}\right)^{l}\frac{1}{l!}\,
%\sum_{m = 0}^{n - 2l - 2k} (-1)^{m} \frac{1}{m!}  \\
 =  \left(\frac{1}{2}\right)^k \frac{1}{k!} \sum_{l = 0}^{\lfloor n/2 \rfloor - k} (-1)^{l} 
    \left(\frac{1}{2}\right)^{l}\frac{1}{l!} \frac{D_{n - 2 l -2k}}{(n - 2 l-2k)!} 
\en
\ignore{
When $k = 0$, this reduces to
\eq\label{prob00}
\mathbb{P}(C_1(n) = 0, C_2(n) = 0) =  \sum_{l = 0}^{\lfloor n/2 \rfloor} (-1)^{l} 
    \left(\frac{1}{2}\right)^{l}\frac{1}{l!}\, \frac{D_{n - 2 l}}{(n - 2 l)!} 
\en
}
(\ref{prob0k}) follows from \cite[Eq. (1.9)]{arratia_logarithmic_2003}, and the well-known (\ref{ecjn*}) is derived in the context of $\theta$-biased random derangements (with $\theta = 1$) in \cite{dasilva2020}. 

\subsection{The number of 2-cycles in the core}
In this section, we look in more detail at the cycles in the core of the mapping.
We begin with  some properties of the number $N_n$ of elements in the core of a standard random mapping. If we define 
$\pi_k = \prod_{l=0}^{k} \left( 1-\frac{l}{n}\right) = (n-1)_{[k]}/n^k$
then $n \pi_k = (n-k) \pi_{k-1}$ for $k = 1, 2, \ldots, n-1.$ Hence 
$\pi_{k-1} - \pi_k = k\pi_{k-1}/n$, and it follows from (\ref{corenumber}) that, for $j = 1, 2, \ldots, n,$ 
\begin{equation}\label{identity}
\mathbb{P}(N_n \geq j) = \sum_{k=j}^n \frac{k \pi_{k-1}}{n} = \sum_{k=j}^n (\pi_{k-1} - \pi_k) = \pi_{j-1} - \pi_n = \pi_{j-1}.
\end{equation}
\ignore{
As a consequence, we recover the well-known fact that 
$$%\eq\label{gammabarmean}
\bbbe N_n  =  \sum_{l=1}^{n} \BP(N_n \geq l) =  \sum_{l=1}^{n} \pi_{l-1} =
\sum_{l=0}^{n-1} \prod_{j=0}^l \left(1 -\frac{j}{n}\right) = \sum_{l=0}^{n-1} \frac{(n-1)_{[l]}}{n^l}.
$$%\en
}

To find the distribution of the number $\tC_2^*(n)$ of 2-cycles in the core, we make use of the following result.
\begin{lemma}\label{lem1}
For any $n \geq 2$ and $m = 1,2,\ldots,n$ we have
\eq\label{needtoshow}
\left(\frac{n}{n-1}\right)^n \,\sum_{r = m}^n \frac{r}{n} \frac{n_{[r]}}{n^r} \frac{D_{r - m}}{(r - m)!} = \frac{n_{[m]}}{ (n-1)^{m}} 
\en
\end{lemma}

\begin{proof}
We have
\begin{eqnarray*}
\sum_{r = m}^n \frac{r}{n} \frac{n_{[r]}}{n^r} \frac{D_{r-m}}{(r-m)!} & = & \sum_{r = m}^n \frac{r}{n} \frac{n_{[r]}}{n^r} \sum_{j = 0}^{r-m} \frac{(-1)^j}{j!}\\
& = & \sum_{j=0}^{n-m} \frac{(-1)^j}{j!} \, \sum_{r = m+j}^n \frac{r}{n} \frac{n_{[r]}}{n^r} \\
& = & \sum_{j=0}^{n-m} \frac{(-1)^j}{j!} \,\BP(N_n \geq m+j) \qquad \textrm{ from }(\ref{corenumber}) \\
& = & \sum_{j=0}^{n-m} \frac{(-1)^j}{j!} \,\frac{n_{[m+j]}}{n^{m+j}} \qquad \textrm{ from }(\ref{identity}) \\
& = & \frac{n_{[m]}}{n^m} \, \sum_{j=0}^{n-m} \frac{(-1)^j}{j!}\,\frac{(n-m)!}{(n-m-j)!}\frac{1}{n^j} \\
& = & \frac{n_{[m]}}{n^m} \,\sum_{j=0}^{n-m} \genfrac(){0pt}{}{n-m}{n-m-j} \left(-\frac{1}{n}\right)^j \\
& = & \frac{n_{[m]}}{n^m} \left(1 - \frac{1}{n}\right)^{n-m}
\end{eqnarray*}
%\binom{n-m}{n-m-j}
It follows that the left side of (\ref{needtoshow}) is
$$
\left(\frac{n}{n-1}\right)^n \, \frac{n_{[m]}}{n^m} \left(\frac{n-1}{n}\right)^{n-m} = \frac{n_{[m]}}{(n-1)^m},
$$
which establishes Lemma~\ref{lem1}.%(\ref{needtoshow}).
\end{proof}

\begin{lemma}\label{lem2}
For $k=0,1,\ldots,\lfloor n/2 \rfloor,$
\eq\label{c2n*law}
\BP(\tC_2^*(n) = k) = \left(\frac{1}{2}\right)^k \frac{1}{k!} \sum_{l = 0}^{\lfloor n/2\rfloor - k} (-1)^l  \left(\frac{1}{2}\right)^l \frac{1}{l!}\,\frac{n_{[2l+2k]}}{(n-1)^{2l + 2k}}.
\en
\end{lemma}
\begin{proof}
 It follows by conditioning that the number $\tN_n$ in the core of the mapping with no singleton cycles has distribution
\begin{eqnarray}
\BP(\tN_n = r)&  = & \BP(N_n = r) \frac{D_r}{r!} \big/ \left( \frac{n-1}{n}\right)^n
\nonumber \\
& = & \left(\frac{n}{n-1}\right)^n \frac{r}{n} \frac{n_{[r]}}{n^r} \frac{D_r}{r!}, r = 2, 3, \ldots,n. \label{toescore}
\end{eqnarray}
The law of the number of 2-cycles in the core is therefore given by
$$%\begin{eqnarray*}
\BP(\tC_2^*(n) = k) =\sum_{r=2k}^n \BP(\tN_n = r) 
 \times \bbbp(\textrm{random derangement of size }r\textrm{ has }k \textrm{ 2-cycles}),
$$%\end{eqnarray*}
and the latter probability can be found via (\ref{prob0k}) with $n$ there replaced by $r$. We obtain, after some simplification,
\begin{eqnarray*}
\BP(\tC_2^*(n) = k) & = & \left(\frac{n}{n-1}\right)^n  \frac{2^{-k}}{k!}\, \sum_{r=2 k}^n \frac{r}{n} \frac{n_{[r]}}{n^r} 
\,  \sum_{l = 0}^{\lfloor r/2 \rfloor - k} (-1)^{l} \left(\frac{1}{2}\right)^{l}\frac{1}{l!}\, \frac{D_{r - 2 l -2k}}{(r - 2 l-2k)!}  \\
& = &  \frac{2^{-k}}{k!}\,\sum_{l = 0}^{\lfloor n/2 \rfloor - k} (-1)^{l} \left(\frac{1}{2}\right)^{l}\frac{1}{l!} \left(\frac{n}{n-1}\right)^n \sum_{r=2l+2k}^n \frac{r}{n} \frac{n_{[r]}}{n^r} \frac{D_{r - 2 l -2k}}{(r - 2 l-2k)!}.
\end{eqnarray*}
The term on the last line reduces to $n_{[2l+2k]}/(n-1)^{2l + 2k}$ by
using the identity in (\ref{needtoshow}), and completing the proof.
\end{proof}
A numerical example is given in Table~\ref{tab3}.

\subsubsection{Expected number of cycles of length $j$}
It is well known that the expected number of cycles of length $j$ in a standard random mapping core is
\eq\label{rmecjn*}
\bbbe C_j^*(n) = \sum_{r = j}^n \frac{r}{n} \frac{n_{[r]}}{n^r} \frac{1}{j} = \frac{1}{j} \BP(N_n \geq j) = \frac{1}{j} \frac{n_{[j]}}{n^j}, \quad j=1,\ldots,n; 
\en
see (\ref{identity}) for the last step. For the screaming toes mapping, the expected number of cycles of length $j$ is, from (\ref{ecjn*}) and (\ref{toescore}), 
\begin{eqnarray} 
\BE \tC_j^*(n) & = & \left(\frac{n}{n-1}\right)^n\, \sum_{r  = j}^n \frac{r}{n} \frac{n_{[r]}}{n^r} \frac{D_r}{r!}\, \frac{1}{j} \frac{r!}{D_r}\frac{D_{r-j}}{(r-j)!} \nonumber \\
& = & \left(\frac{n}{n-1}\right)^n\, \frac{1}{j} \sum_{r = j}^n \frac{r}{n} \frac{n_{[r]}}{n^r} \frac{D_{r-j}}{(r-j)!} \nonumber \\
& = & \frac{1}{j} \frac{n_{[j]}}{(n-1)^j}, \label{ecjn*toes}
\end{eqnarray}
the final equality coming from (\ref{needtoshow}). Some numerical values are given in Table~\ref{tab1a}.

\begin{remark}
Since the number of components is equal to the number of cycles in the core, $\bbbe \tK_n$ may also be computed from (\ref{ecjn*toes}), to obtain
\eq\label{ekn2}
\bbbe \tK_n = \sum_{j=2}^n \frac{1}{j} \frac{n_{[j]}}{(n-1)^j},
\en
which should be compared to (\ref{ekn}). The equivalence of (\ref{ekn}) and (\ref{ekn2}) is illustrated for the case of $n = 10$ in Tables~\ref{tab2} and \ref{tab1a}.
\end{remark}

\subsection{Did anyone scream?}

Cameron's original problem was to show that the  probability that someone screams~is
\begin{equation}\label{pcresult}
q_n := \sum_{l = 1} ^ {\lfloor n/2 \rfloor} \frac{(-1)^{l-1} n_{[2l]}}{2^l l! (n-1)^{2l}},
\end{equation}
 and to find the limiting behavior of $q_n$ as $n \to \infty$. 
% where $n_{(j)} = n(n-1) \cdots (n-j+1)$,
We can identify $q_n$ because $\BP({\rm someone\ screams}) =  1 - \BP(\tC_2^*(n) = 0),$ 
so from Lemma~\ref{lem2},  
$$%\eq\label{noscreamprob}
q_n = \BP(\tC_2^*(n) > 0) =  \sum_{l = 0}^{\lfloor n/2\rfloor} (-1)^{l-1} \left(\frac{1}{2}\right)^l \frac{1}{l!} \,\frac{n_{[2l]}}{(n-1)^{2l}},
$$%\en
recovering (\ref{pcresult}). Representative values of $q_n$ are given in Table \ref{tab1}.

\begin{table}[htbp]
\begin{center}
\begin{tabular}{|c|c||c | c |}
\hline
$n$ & $q_n$ & $n$ & $q_n$ \\
\hline
5 & 0.5664 & 60 & 0.4039 \\
10 & 0.4654 & 70 & 0.4023 \\
15 & 0.4386 & 80 & 0.4012 \\
20 & 0.4264 & 90 & 0.4003\\
30 & 0.4148 & 100 & 0.3996 \\
40 & 0.4093 & 1,000 & 0.3941\\
50 & 0.4060 & 10,000 & 0.3935\\
\hline
\end{tabular}
\caption{The probability $q_n$ from (\ref{pcresult}) of at least one screaming pair for various values  of $n$.}\label{tab1}
\end{center}
\end{table}

Finally, a word about the limiting value of $q_n$. It is straightforward to show that the joint law of $(\tC^*_2(n),\tC^*_3(n),\ldots)$ converges to that of independent Poisson random variables, the $j$th of which has mean $1/j$, just as in the standard mapping case.
In particular, 
$$%\eq\label{qnlim}
\lim_{n \to \infty} q_n = 1 - \bbbp({\rm Po}(1/2) = 0) = 1 - e^{-1/2} \approx 0.3935.
$$%\en

\ignore{
To establish the asymptotics of (\ref{pcresult}), we can use () and (\ref{toescore}) to see that for any fixed $b \leq n$, and $r_2, r_3, \ldots, r_b \geq 0$ satisfying $2 r_2+\cdots + br_b = m \leq n$,
\eq\label{mapcoremeans}
\mathbb{E}(C^*_2(n)^{[r_2]} \cdots C^*_b(n)^{[r_b]}) = \prod_{j=2}^b \left(\frac{1}{j}\right)^{r_j}\, \frac{n_{(m)}}{(n-1)^m},
\en
so that $C_2^*(n),\ldots,C_b^*(n))$ converges in distribution to independent Poisson components $(Z_2,\ldots,Z_n)$ where $\BE Z_i = 1/i$. (This result also follows directly from () $N_n^* \to \infty$ in probability.) Hence $q_n = \BP(C_2^*(n) > 0) \to 1 - e^{-1/2}$, as claimed in (\ref{qnlim}).
}

\section{Simulating the component counts}%\label{simulsect}

It is often useful to be able to simulate combinatorial objects, for example to study the distributions of cycle lengths and component sizes for moderate values of $n$, where the asymptotics might not be good, or when asking more detailed questions where explicit answers are hard to come by. In our setting, there are (at least) two approaches to this. 

\subsection{A rejection method}\label{rejmethod}

The first is useful for studying the component counting process $(\tC_2(n),\tC_2(n),\ldots)$, by exploiting a modification of the simulation approach in~\cite[Section 4]{arratia_simulating_2018}. For any $\theta > 0$, (\ref{jointlaw}) gives the distribution of $(\tC_2(n),\ldots,\tC_n(n))$  as
\begin{eqnarray*}
\bbbp(\tC_j(n) = a_j, j=2,\ldots,n) & \propto & \bbbone\left\{ \sum_{j=2}^n j a_j = n\right\}\, \prod_{j=2}^n \left(\frac{\omega_j}{j}\right) ^{a_j} \frac{1}{a_j!} \nonumber \\
& = & \bbbone\left\{ \sum_{j=2}^n j a_j = n\right\}\, \prod_{j=2}^n \left(
\frac{\omega_j}{\theta}\right)^{a_j}\, \prod_{j=2}^n  \left(
\frac{\theta}{j}\right)^{a_j} \frac{1}{a_j!} \nonumber \\
& = &  \prod_{j=2}^n \left(
\frac{\omega_j}{\theta}\right)^{a_j}\, \bbbone\left\{ \sum_{j=2}^n j a_j = n\right\}\,\prod_{j=2}^n  \left(
\frac{\theta}{j}\right)^{a_j} \frac{1}{a_j!},
\end{eqnarray*}
where $\omega_j = j \tlam_j = \bbbp({\rm Po}(j) < j-1)$ is given by (\ref{lambdajhat}). The last factorization shows that we can simulate $(a_2,\ldots,a_n)$ from the distribution
\eq\label{esfcondlaw}
\BP_\theta(a_2,\ldots,a_n) \propto \bbbone\left\{ \sum_{j=2}^n j a_j = n\right\}\,\prod_{j=2}^n  \left(
\frac{\theta}{j}\right)^{a_j} \frac{1}{a_j!}
\en
and, assuming $\theta$ can be chosen to make $\omega_j \leq \theta$ for $j = 2, 3, \ldots,n$, 
we accept $(a_2,a_3,\ldots,a_n)$ as an observation from the required distribution with probability
$$%\eq\label{accrate}
h(a_2,\ldots, a_n) = \prod_{j=2}^n \left( \frac{\omega_j}{\theta} \right)^{a_j}.
$$%\en
The method relies on efficient simulation from (\ref{esfcondlaw}), which is precisely that of the counts $(C_1(n),\ldots,C_n(n))$ with the Ewens Sampling Formula with parameter $\theta$, conditional on $C_1(n) = 0$. For more information on this law in the context of $\theta$-biased derangements, see ~\cite{dasilva2020}.

From (\ref{lambdajhat}) we note that we may take $\theta = 1/2$. We implement the algorithm in a slightly different way, by simulating $(a_1,\ldots,a_n)$ from the regular Ewens Sampling Formula with parameter $\theta = 1/2$, and accepting $(a_2,\ldots,a_n)$ with probability 
\eq\label{accrate}
h(a_1, a_2,\ldots, a_n) = \bbbone(a_1 = 0) \prod_{j=2}^n \left( \frac{\omega_j}{\theta} \right)^{a_j}.
\en

There are many ways to generate observations from the Ewens Sampling Formula with an arbitrary parameter $\theta$, for example by using the Chinese Restaurant Process or the Feller Coupling; we exploit the latter, and point the reader to the discussion in \cite{arratia_simulating_2018} about the efficiency of these methods. 

%In the code in the {\sf R} appendix, 
%We also note that we could arrange the simulation a little differently, by simulating an observation from the ESF conditional on it having no singleton cycles, and then modifying the acceptance probability to $h(c_2,\ldots, c_n) =  \prod_{j=2}^n \left( \frac{\omega_j}{\theta} \right)^{c_j}.$ The first approach makes it somewhat simpler to estimate the overall acceptance rate.
\subsection{Estimating the acceptance probability}

To compute the asymptotic acceptance probability, we note that 
$$
\mathbb{P}(\mbox{accept an observation}) = \mathbb{E} \left( \bbbone(C_1(n) = 0)\,  \prod_{j=2}^n ( 2 \omega_j)^{C_j(n)} \right),
$$
where  $(C_1(n),\ldots,C_n(n))$ has the Ewens Sampling Formula with parameter $\theta = 1/2$. The Poisson limit heuristic shows that this is asymptotically
\begin{eqnarray}
e^{-1/2}\,\prod_{j=2}^\infty \mathbb{E} (2\omega_j)^{Z_j} 
& = &  e^{-1/2}\, \prod_{j=2}^\infty \exp\left( - \frac{1}{2j} \left( 1 - 2\omega_j \right) \right) \nonumber \\
& = & e^{- 1/2}\, \exp \left( - \sum_{j = 2}^\infty \frac{1}{j} \left(\frac{1}{2} -  \bbbp({\rm Po}(j) < j-1)\right)\right)  \nonumber \\
& = & \frac{1}{e} \,\frac{1}{\sqrt{2}}, \label{accrate1}
\end{eqnarray}
the last result following because the algorithm is effectively generating a standard random mapping, and accepting that mapping if its core is a derangement; from (\ref{toescore}), this has asymptotically probability $1/e$. In $10^6$ simulations of the case $n = 10$ the acceptance rate was estimated to be 0.247, in reasonable agreement with the limiting value of $\approx 0.260$ from (\ref{accrate}).

\begin{remark}
There is an appealing connection between the exponent in the right-hand term in (\ref{accrate1}),
\eq\label{rhterm}
\sum_{j = 2}^\infty \frac{1}{j} \left(\frac{1}{2} -  \bbbp({\rm Po}(j) < j-1)\right)
\en
and Spitzer's Theorem~\cite{fs1956}, which is described in detail in \cite[Theorem 1, p.612]{feller1970}. This may be used to evaluate the corresponding value for a \emph{standard} mapping, 
$$
\sum_{j = 1}^\infty \frac{1}{j} \left(\frac{1}{2} -  \bbbp({\rm Po}(j) < j)\right) = \frac{1}{2} \log 2,
$$
as given for example in \cite[Eqn. (19)]{arratia_simulating_2018} and \cite{dep91}. 
It follows that (\ref{rhterm}) is 
\begin{eqnarray}
\lefteqn{
\frac{1}{2} \log 2 -\left(\frac{1}{2} - \frac{1}{e}\right) + \sum_{j=2}^\infty \frac{1}{j} \bbbp({\rm Po}(j)=j-1) }&& \nonumber \\
&= &  \frac{1}{2} \log 2 -\left(\frac{1}{2} - \frac{1}{e}\right) + 1 - \frac{1}{e} = \frac{1}{2} (1 + \log 2), \label{jfck}
\end{eqnarray}
the sum being the probability that a Borel distribution with parameter 1 is at least 2.
This provides the formal justification of (\ref{accrate1}).
%I have been unable to construct the right random walk to simplify the evaluation of (\ref{rhterm}).
\end{remark}

\subsection{Simulating component and core sizes}\label{jointmethod}

A rejection method can be used to study details of the cycle sizes in the core of a  mapping, by generating an observation $r$ from the distribution of $\tN_n$ in (\ref{toescore}), and then generating a random derangement of size $r$. While we do not illustrate this approach here, see~\cite{dasilva2020} for efficient methods for generating $\theta$-biased derangements. 

The second approach simulates a random mapping with no singleton cycles in its core, as determined by the random variables in (\ref{Bdef1}), and processes the output using (for example) the {\sf R} igraph package \cite{csardi2006} to compute the component and core sizes. This provides a computationally cheap way to check the first approach, and provides a way to study aspects of the joint law of component and cycle sizes. 

\subsection{Examples}
Here we illustrate some of the explicit results obtained above, and their corresponding simulated values, all in the setting of $n = 10$. Table~\ref{tab2} compares the mean number of components for the screaming toes mapping with the corresponding values for the regular mapping. The mean number of components is 1.251 for the screaming toes mapping, and 1.913 for the standard mapping.

\begin{table}[hbtp]
\begin{center}
\begin{tabular}{|c|cc||c|}
\hline
  & $\bbbe \tC_j(10)$  & Simulation &  $\bbbe C_j(10)$\\
$j$ & (\ref{cjmeans}) & &  (\ref{hsfmeans}) \\
\hline
  1 &         &        & 0.3874 \\
  2 &  0.0744 & 0.0745 & 0.2265 \\
  3 &  0.0771 & 0.0764 & 0.1680 \\
  4 &  0.0734 & 0.0734 & 0.1391 \\
  5 &  0.0699 & 0.0699 & 0.1235 \\
  6 &  0.0673 & 0.0676 & 0.1160 \\
  7 &  0.0654 & 0.0650 & 0.1150 \\
  8 &  0.0608 & 0.0607 & 0.1225 \\
  9 &  0.0000 & 0.0000 & 0.1489 \\
 10 &  0.7629 & 0.7633 & 0.3660 \\
 \hline
%Mean number of components =  1.251137 \\
\end{tabular}
\caption{Mean number of components of sizes 1(1)10 for $n = 10$ for the screaming toes mapping, simulated values from $10^6$  realizations using the method in Section~\ref{rejmethod}, and the corresponding means for a standard mapping. }\label{tab2}
\end{center}
\end{table}

Table~\ref{tab3} illustrates the distribution of the number of screaming pairs, while Table~\ref{tab1a} compares the mean cycle counts for the screaming toes core, and the corresponding values for the standard core.
The mean number of cycles is 1.251 for the screaming toes mapping, and 1.913 for the standard mapping, the former in agreement with the results in (\ref{ekn}) and (\ref{ekn2}). Table~\ref{tab5} illustrates the distribution of the number of elements in the core.

\begin{table}[htbp]
\begin{center}
\begin{tabular}{|c|c|c|}
\hline
 & $\BP(\tC_2^*(10) = k)$ & Simulation \\
$k$ &  (\ref{c2n*law}) & \\
\hline
  0 &  0.5346 &  0.5352 \\
  1 &  0.3809 &  0.3800 \\
  2 &  0.0789 &  0.0791 \\
  3 &  0.0055 &  0.0056 \\
  4 &  0.0001 &  0.0001 \\
  5 &  0.0000 &  0.0000 \\
 \hline
\end{tabular}
\caption{Distribution of the number of screaming pairs, from (\ref{c2n*law}), for $n = 10$. Simulated values from $10^6$  realizations of the method in Section~\ref{jointmethod}.}\label{tab3}
\end{center}
\end{table}
 % from simulation: [1] 535195 379953  79132   5616    104  0  0  0 0  0  0

\begin{table}[htbp]
\begin{center}
\begin{tabular}{|c|cc||c|}
\hline
 & $\bbbe \tC_j^*(10)$  & Simulation & $\bbbe C_j^*(10)$\\
 $j$ &(\ref{ecjn*toes}) & &  (\ref{rmecjn*})\\
\hline
  1 &         &        &  1.0000\\
  2 &  0.5555 & 0.5555 &  0.4500\\
  3 &  0.3292 & 0.3292 &  0.2400\\
  4 &  0.1923 & 0.1920 &  0.1260\\
  5 &  0.1029 & 0.1024 &  0.0605\\
  6 &  0.0472 & 0.0474 &  0.0252\\
  7 &  0.0182 & 0.0181 &  0.0086\\
  8 &  0.0054 & 0.0053 &  0.0023\\
  9 &  0.0010 & 0.0010 &  0.0004\\
 10 &  0.0001 & 0.0001 &  0.0000\\
 \hline
%Mean number of components =  1.251137 \\
\end{tabular}
\caption{Mean number of cycles of sizes 1(1)10 in the core for $n = 10$ for the screaming toes mapping, simulated values from $10^6$  realizations of the method in Section~\ref{jointmethod}, and the corresponding means for a typical random mapping. }\label{tab1a}
\end{center}
\end{table}
% values in third column
%[1] 1.000000000 0.450000000 0.240000000 0.126000000 0.060480000 0.025200000 0.008640000 0.002268000
%[9] 0.000403200 0.000036288 1.913027488

\begin{table}[htbp]
\begin{center}
\begin{tabular}{|c|cc||c|}
\hline
 & $\bbbp(\tN_{10} = r)$  & Simulation & $\bbbp(N_{10} = r)$\\
 $r$ &(\ref{toescore}) & &  (\ref{corenumber})\\
\hline
  1 &         &        &  1.0000\\
  2 &  0.2581 & 0.257 &  0.1000\\
  3 &  0.2065 & 0.206 &  0.1800\\
  4 &  0.2168 & 0.217 &  0.2016\\
  5 &  0.1590 & 0.159 &  0.1512\\
  6 &  0.0958 & 0.096 &  0.0907\\
  7 &  0.0447 & 0.045 &  0.0423\\
  8 &  0.0153 & 0.015 &  0.0145\\
  9 &  0.0034 & 0.003 &  0.0033\\
 10 &  0.0004 & 0.0003 &  0.0004\\
 \hline
\end{tabular}
\caption{Probability distribution of the number of elements in the core for $n = 10$ for the screaming toes mapping, simulated values from $10^6$  realizations of the method in Section~\ref{jointmethod}, and the corresponding probabilities for a typical random mapping. }\label{tab5}
\end{center}
\end{table}

% screaming toes  map: 0.0000000000 0.2581174792 0.2064939833 0.2168186825 0.1590003672 0.0957615848 0.0446646486 0.0153146263 0.0034457651 0.0003828631
% regular map 0.10000000 0.18000000 0.21600000 0.20160000 0.15120000 0.09072000 0.04233600 0.01451520 0.00326592 0.00036288
% simulated counts from op9A
% 257317 206227 217129 159098  95939  45100  15373   3479    338 

% Put in example about repeated component lengths, cycle lengths
As a final example, we estimate, from  $10^6$ realizations of the simulation method in Section~\ref{jointmethod}, the probability that the screaming toes mapping with $n  10$ has no repeated component sizes  to be 0.959, no repeated cycle sizes to be 0.898, and no repeated component or cycle sizes to be 0.879. 
%For comparison, for a standard mapping of size $n = 10$, the probability of no repeated component sizes is ?

\section{Discussion}
This paper has focused on the small components and cycles of the screaming toes mapping because that is where the main difference with the standard case emerge. We noted before
(\ref{lambdajhat}) that the screaming toes mapping is logarithmic, in that the Poisson random variables in (\ref{t1ndef}) satisfy
$$
i \bbbp(Z_i = 1) \to 1/2, \textrm{ and } i \bbbe Z_i \to 1/2 \textrm{ as } i \to \infty,
$$
this following from (\ref{lambdajhat}). As a consequence (\cite[Chapter 6]{arratia_logarithmic_2003}) the largest component sizes, when scaled by $n$, have asymptotically the Poisson-Dirichlet law with parameter $\theta = 1/2$, just as in a standard mapping. In a similar vein, the largest cycle lengths, when scaled by $\tN_n$, have asymptotically the Poisson-Dirichlet distribution with parameter $\theta = 1$, once more just as for the standard mapping core.

\ignore{
\begin{itemize}
\item I have tried to motivate the approach through its inherently probabilistic nature, rather than by ``counting" methods. It's a choice!
\item Could add central limit results, e.g. for number of components, if add things about total variation estimates to the independent Poisson limits. Probably a coupling proof of this with a regular r.m.?
\item It would be worth a look at joint laws of the component size counts and the cycle counts. Perhaps worth a separate look, even in the standard case?
\item It would be simple to add things about big components and the Poisson-Dirichlet laws.
\end{itemize}
}

\bibliography{amsrefs1}

\begin{thebibliography}{10}

\bibitem{arratia_simulating_2018}
{\sc Arratia, R., Barbour, A., Ewens, W. and Tavar{\'e}, S.} (2018).
\newblock Simulating the component counts of combinatorial structures.
\newblock {\em Theoretical Population Biology\/} {\bf 122,} 5--11.

\bibitem{arratia_logarithmic_2003}
{\sc Arratia, R., Barbour, A. and Tavar{\'e}, S.} (2003).
\newblock {\em Logarithmic combinatorial structures: a probabilistic approach}.
\newblock European Mathematical Society Publishing House, Zuerich, Switzerland.

\bibitem{bb85}
{\sc Bollob\'as, B.} (1985).
\newblock {\em Random Graphs}.
\newblock Academic Press, New York.

\bibitem{cameron_notes_2017}
{\sc Cameron, P.~J.} (2017).
\newblock {\em Notes on {Counting}: {An} {Introduction} to {Enumerative}
  {Combinatorics}}.
\newblock Cambridge University Press, Cambridge.

\bibitem{csardi2006}
{\sc Csardi, G. and Nepusz, T.} (2006).
\newblock The igraph software package for complex network research.
\newblock {\em InterJournal Complex Systems\/} 1695.

\bibitem{dasilva2020}
{\sc {da Silva}, P.~H., Jamshidpey, A. and Tavar\'e, S.} (2020).
\newblock Random derangements and the {Ewens Sampling Formula}.
\newblock Submitted.

\bibitem{dep91}
{\sc Donnelly, P., Ewens, W.~J. and Padmadisastra, S.} (1991).
\newblock Random functions: exact and asymptotic results.
\newblock {\em Adv. Appl. Prob.\/} {\bf 23,} 437--455.

\bibitem{feller1970}
{\sc Feller, W.} (1970).
\newblock {\em An introduction to probability theory and its applications}
  2nd~ed. vol.~II.
\newblock John Wiley \& Sons, Inc., New York.

\bibitem{bh60}
{\sc Harris, B.} (1960).
\newblock Probability distributions related to random mappings.
\newblock {\em Ann. Math. Stat.\/} {\bf 31,} 1045--1062.

\bibitem{kolchin1976}
{\sc Kolchin, V.~F.} (1976).
\newblock A problem of allocation of particles in cells and random mappings.
\newblock {\em Theory Probab. its Appl.\/} {\bf 21,} 48--63.

\bibitem{fs1956}
{\sc Spitzer, F.} (1956).
\newblock A combinatorial lemma and its application to probability theory.
\newblock {\em Trans Amer Math Soc\/} {\bf 82,} 323--339.

\end{thebibliography}
\bibliographystyle{APT}

\noindent {\bf Acknowledgements: } I thank John Kingman for comments that led to the evaluation in (\ref{jfck}). 
{\sf R} code for performing the computations described in the paper may be obtained from the author.

\end{document}